\documentclass[12pt,reqno]{amsart}

\usepackage{amssymb}

\newtheorem{theorem}{Theorem}[section]
\newtheorem{proposition}[theorem]{Proposition}
\newtheorem{lemma}[theorem]{Lemma}

\theoremstyle{definition}
\newtheorem{definition}[theorem]{Definition}

\numberwithin{equation}{section}

\begin{document}

\baselineskip=15pt

\title[Branched holomorphic Cartan geometry on Sasakian manifolds]{Branched holomorphic Cartan 
geometry on Sasakian manifolds}

\author[I. Biswas]{Indranil Biswas}

\address{School of Mathematics, Tata Institute of Fundamental
Research, Homi Bhabha Road, Mumbai 400005, India}

\email{indranil@math.tifr.res.in}

\author[S. Dumitrescu]{Sorin Dumitrescu}

\address{Universit\'e C\^ote d'Azur, CNRS, LJAD, France}

\email{dumitres@unice.fr}

\author[G. Schumacher]{Georg Schumacher}

\address{Fachbereich Mathematik und Informatik,
Philipps-Universit\"at Marburg, Lahnberge, Hans-Meerwein-Strasse, D-35032
Marburg, Germany}

\email{schumac@mathematik.uni-marburg.de}



\date{}

\begin{abstract}
We extend the notion of (branched) holomorphic Cartan geometry on a complex manifold to the
context of Sasakian manifolds. Branched holomorphic Cartan geometries on Sasakian Calabi-Yau
manifolds are
investigated.
\end{abstract}

\maketitle

\section{Introduction}

Contact manifolds are the odd dimensional counterparts of the symplectic manifolds. Just as 
the total space of the cotangent bundle of a $C^\infty$ manifold $M$ is a typical (local) model 
of a symplectic manifold, the total space of the projective bundle $P(T^*M)$ is a typical 
(local) model of a contact manifold. In a similar vein, Sasakian manifolds are the odd 
dimensional counterparts of the K\"ahler manifolds. A compact regular Sasakian manifold is the 
unit circle bundle inside a holomorphic Hermitian line bundle of positive curvature on a complex 
projective manifold. More generally, a compact quasi-regular Sasakian manifold is the unit circle 
bundle inside a holomorphic Hermitian line bundle of positive curvature on a K\"ahler orbifold. 
The global structure of compact irregular Sasakian manifolds does not admit
any such explicit description. 
Sasakian manifolds were introduced by Sasaki \cite{Sa}, \cite{SH}, which explains the 
terminology. In last fifteen years, there has been a substantial increase of the interest in 
Sasakian manifolds (see \cite{BG2} and references therein). It is evident from the references 
that a very large part of this recent investigations into Sasakian manifolds were carried out by 
C. Boyer and K. Galicki. Another aspect of this recent activities in Sasakian manifolds is the 
discovery of their relevance in string theory. This was initiated in the works of J. Maldacena 
\cite{Mal}. For further developments in this direction see \cite{GMSW}, \cite{MSY}, \cite{MS} and 
references therein.

Let $G$ be a complex Lie group and $H\, \subset\, G$ a complex Lie subgroup. A holomorphic 
Cartan geometry of type $(G,\, H)$ is a complex manifold equipped with an infinitesimal 
structure that mimics the infinitesimal structure of the quotient manifold $G/H$. In particular, 
a flat Cartan geometry of type $(G,\, H)$ is a complex manifold equipped with local charts 
modeled on open subsets of $G/H$ such that all the transition functions are given by the 
elements of $G$ acting on $G/H$ as left--multiplication. See \cite{Sh} and references therein 
for Cartan geometry. It should be mentioned that Cartan geometry has close ties with theoretical 
physics; see \cite{RS}, \cite{Ha}, \cite{AFL}, \cite{Ho} and references therein. Being motivated 
by the works of Mandelbaum \cite{Ma1}, \cite{Ma2}, the notion of a holomorphic Cartan geometry 
was enhanced to branched holomorphic Cartan geometry; this was done in \cite{BD}.

Our aim here is to develop an analog of holomorphic Cartan geometries and of branched 
holomorphic Cartan geometries in the context of Sasakian manifolds. Most of our efforts were 
spent in building the foundations. In the last section we investigate holomorphic Cartan 
geometries on compact quasi-regular Sasakian manifolds that are Calabi-Yau. We hope in future to 
investigate further this topic of holomorphic Cartan geometries on Sasakian manifolds.

\section{Holomorphic principal bundles on Sasakian manifolds}

\subsection{Sasakian manifolds}

Let $X$ be a connected oriented smooth real manifold of dimension
$2m+1$, where $m$ is a positive integer, and let $g$ be a $C^\infty$ Riemannian
metric on $X$. The Levi-Civita connection on the real tangent bundle $TX$ associated to $g$
will be denoted by $\nabla$.

\begin{definition}[{\cite[Definition-Theorem 10]{BGsusy}}]\label{de:sasaki}
The above Riemannian manifold $(X,\, g)$ is called a \textit{Sasakian manifold} if any of the
following three equivalent conditions hold:
\begin{enumerate}
\item[(i)] There is a Killing vector field $\xi$ on $X$ of length one such that the section
\begin{equation}\label{Phi}
\Phi \, \in\, C^\infty(X,\, TX\otimes (TX)^*)
\end{equation}
defined by $\Phi (v) \,=\, -\nabla_v\xi$ satisfies the identity
\begin{equation}\label{id.}
(\nabla_v \Phi) (w)\, =\, g(v\, ,w)\xi- g(\xi ,\, w)v
\end{equation}
for all $v\, ,w\,\in\, T_xX$ and all $x\,\in\, X$.

\item[(ii)] There is a Killing vector field $\xi$ on $X$ of unit length such
that the Riemann curvature tensor $R$ of $(X,\, g)$ satisfies the identity
$$
R(v,\, \xi)w \, =\, g(\xi,\, w)v- g(v,\, w)\xi
$$
for all $v$ and $w$ as above.

\item[(iii)] The metric cone $({\mathbb R}_+\times X,\, dr^2 \oplus r^2\cdot g)$ is K\"ahler.
\end{enumerate}
\end{definition}

Given a Killing vector field $\xi$ of
unit length satisfying condition (i) in Definition \ref{de:sasaki}, the K\"ahler structure on
${\mathbb R}_+\times X$ asserted in statement (iii) is constructed as follows. Let $F$ be
the distribution on $X$ of rank $2m$ given by the orthogonal complement $\xi^\perp$ of
$\xi$. The homomorphism $\Phi$ (defined in \eqref{Phi}) preserves $F$, and
furthermore,
\begin{equation}\label{e0}
(\Phi\vert_F)^2 \, =\, -\text{Id}_F\, .
\end{equation}
Let $\widetilde{J}$ be the almost complex structure on ${\mathbb R}_+
\times X$ defined as follows:
$$
\widetilde{J}\vert_F\,=\, \Phi\vert_F\, ,\ \
\widetilde{J}\left(\frac{d}{dr}\right) \,=\, \xi\, ,\ \
\widetilde{J}(\xi) \,=\, -\frac{d}{dr}\, .
$$
This almost complex structure is in fact integrable. The Riemannian metric
$dr^2 \oplus r^2\cdot g$ on ${\mathbb R}_+\times X$ is K\"ahler with respect
to this complex structure $\widetilde{J}$ \cite{BGsusy}, \cite{BG2}.

Conversely, if the metric cone $({\mathbb R}_+\times X,\, dr^2 \oplus r^2\cdot g)$
is K\"ahler, then consider the vector field on ${\mathbb R}_+\times X$ given by
$J(\frac{d}{dr})$, where $J$ is the almost complex
structure on ${\mathbb R}_+\times X$. The vector field $\xi$ on $X$ obtained by
restricting this vector field to $X\, =\, \{1\}\times X\, \subset\, {\mathbb R}_+\times X$
satisfies both condition $(i)$ and $(ii)$ in
Definition~\ref{de:sasaki}, with respect to the induced Riemannian metric $g$.

As the three conditions in Definition \ref{de:sasaki} are equivalent,
the vector field $\xi$ and the K\"ahler structure on ${\mathbb R}_+\times X$ will be considered 
as part of the definition of a Sasakian manifold.

A connected Sasakian manifold $(X,\, g,\, \xi)$ with $X$ a compact
manifold is called \textit{quasi-regular} if all 
the orbits of the unit vector field $\xi$ are closed. It $(X,\, g,\, \xi)$ is not quasi-regular, 
then it is called an \textit{irregular} Sasakian manifold. A quasi-regular connected 
Sasakian manifold $(X,\, g,\, \xi)$ is called \textit{regular} if the vector field $\xi$
integrates into a free and faithful action of $S^1\, =\, {\rm U}(1)$ on $X$.

We refer the reader to \cite{BG2} for Sasakian manifolds.

\subsection{Smooth principal bundles}

Let $M$ be a $C^\infty$ manifold equipped with a $C^\infty$ distribution $S\, \subset\, TM$ of
rank $r$. Let $H$ be a Lie group and $p\, :\, E_H\, \longrightarrow\, M$ a $C^\infty$ principal
$H$--bundle on $M$. A {\it partial connection} on $E_H$ in the direction of $S$ is a
$C^\infty$ distribution
$$
\widetilde{S}\, \subset\, TE_H
$$
of rank $r$ such that
\begin{enumerate}
\item $\widetilde{S}$ is preserved by the action of $H$ on $E_H$, and

\item the differential $dp\, :\, TE_H\, \longrightarrow\, p^*TM$ of $p$ restricts to an isomorphism
between $\widetilde{S}$ and the subbundle $p^*S\, \subset\, p^*TM$.
\end{enumerate}
A partial connection $\widetilde{S}$ is called \textit{integrable} if the distribution
$\widetilde{S}$ is integrable. Note that if $\widetilde{S}$ is integrable, then $S$ must also
be integrable.

If $S\,=\, TM$, then a partial connection is a usual connection.

In general, $S$ can be a complex distribution, meaning a subbundle of the complex
vector bundle $TM\otimes_{\mathbb R}\mathbb C$. Note that the Lie bracket operation on the
locally defined vector fields on $M$ extends to a Lie bracket operation on the locally
defined smooth sections of $TM\otimes_{\mathbb R}\mathbb C$. The notion of a partial connection
extends to complex distributions in an obvious way.

Let $\rho\,:\, H\, \longrightarrow\, Q$ be a homomorphism of Lie groups. Let
$$
E_Q\, :=\, E_H\times^\rho Q\, \longrightarrow\, M
$$
be the principal $Q$--bundle over $X$ obtained by extending the structure group of $E_H$
using the above homomorphism $\rho$. A partial connection $\widetilde S$ on $E_H$ produces
a partial connection on $E_Q$ for the same distribution $S$
on $M$. To see this, we recall that $E_Q$
is the quotient of $E_H\times Q$ where two points $(e_1,\, q_1)$ and $(e_2,\, q_2)$ of
$E_H\times Q$ are identified if there is an element $h\, \in\, H$ such that $e_2\,=\, e_1h$
and $q_2\,=\, \rho(h^{-1})q_1$. Let ${\widetilde S}'$ be the distribution of rank $r$ on
$E_H\times Q$ given by the distribution $\widetilde S$ on $E_H$. More precisely,
for any $(e_1,\, q_1)\, \in\, E_H\times Q$, the subspace
$$
{\widetilde S}'(e_1,\, q_1)\, \subset\, T_{(e_1,\, q_1)}E_H\times Q\,=\,
T_{e_1}E_H \oplus T_{q_1}Q
$$
is ${\widetilde S}(e_1)\, \subset\, T_{e_1}E_H$. This distribution ${\widetilde S}'$ descends
to a distribution on the quotient space $E_Q$ of $E_H\times Q$ by the quotient map. The
resulting distribution on $E_Q$ is in fact a partial connection on $E_Q$ for the distribution $S$.

Let ${\mathcal X}\, :=\, (X,\, g,\, \xi)$ be a connected Sasakian manifold. A
principal $H$--bundle on $\mathcal X$ is defined to be a $C^\infty$ principal $H$--bundle
on $X$ equipped with a partial connection $\widetilde{\xi}$ for the one dimensional distribution
${\mathbb R}\cdot\xi$ on $X$.

For a principal bundle on $X$ with structure group a complex Lie group $H$, it is always
assumed that the fibers of the principal
bundle are complex manifolds and the action of $H$ on the principal bundle preserves the
complex structure of the fibers of the principal bundle. To explain this condition, 
let $H$ be a complex Lie group. Let $p\, :\, E_H\, \longrightarrow\, X$ be
principal $H$--bundle. This implies that
\begin{itemize}
\item the subbundle ${\rm kernel}(dp)\, \subset\, TE_H$
is equipped with a $C^\infty$ automorphism $J_E\, :\, {\rm kernel}(dp)\,\longrightarrow\,
{\rm kernel}(dp)$ such that $J_E\circ J_E\,=\, -\text{Id}$,

\item{} for every $x\, \in\, X$, the
almost complex structure $J_E\vert_{p^{-1}(x)}$ on the fiber $p^{-1}(x)$ is integrable, and

\item the action of $H$ on $p^{-1}(x)$ is holomorphic for every $x\, \in\, X$.
\end{itemize}
Consider the differential $dp\, :\, TE_H\, \longrightarrow\, p^*TX$ of the projection $p$.
Using $J_E$ we get a decomposition of the complex vector bundle
$${\mathbb K}\, :=\, {\rm kernel}(dp)\otimes_{\mathbb R} {\mathbb C}\, \longrightarrow\, X$$
as follows. Define
\begin{equation}\label{f01}
{\mathbb K}^{0,1}\,:=\, \{v+\sqrt{-1}\cdot J_E(v)\, \mid\, v\, \in\, {\rm kernel}(dp)\}
\, \subset\, {\mathbb K}\, :=\, {\rm kernel}(dp)\otimes_{\mathbb R} {\mathbb C}
\end{equation}
and
\begin{equation}\label{f02}
{\mathbb K}^{1,0}\,:=\, \{v-\sqrt{-1}\cdot J_E(v)\, \mid\, v\, \in\, {\rm kernel}(dp)\}
\, \subset\, {\mathbb K}\, ,
\end{equation}
so we have ${\mathbb K}\,=\, {\mathbb K}^{0,1}\oplus {\mathbb K}^{1,0}$ and
${\mathbb K}^{1,0}\,=\, \overline{{\mathbb K}^{0,1}}$.

Let ${\mathcal X}\, :=\, (X,\, g,\, \xi)$ be a connected Sasakian manifold.
Let $(E_H,\, p,\, J_E)$ be a principal $H$--bundle on $X$, where $H$ is a complex
Lie group. Let $\widetilde{\xi}\, \subset\, TE_H$ be a partial connection 
on $E_H$ for the one dimensional distribution
${\mathbb R}\cdot\xi$ on $X$. Consider the (unique) vector field $\widehat{\xi}$ on $E_H$
such that
\begin{itemize}
\item $\widehat{\xi}(z) \, \subset\, \widetilde{\xi}_z$ for all $z\, \in\, E_H$, and also

\item $dp (\widehat{\xi}(z))\,=\, \xi(p(z))$.
\end{itemize}
Note that the flow on $E_H$ associated to the above vector field $\widehat\xi$ takes a fiber
of $p$ to another fiber of $p$. Therefore, the Lie derivative
$L_{\widetilde{\xi}} J_E$ is a $C^\infty$ endomorphism of ${\rm kernel}(dp)$ that anti-commutes
with $J_E$, meaning $J_E\circ (L_{\widetilde{\xi}} J_E)+(L_{\widetilde{\xi}} J_E)\circ J_E\,=\, 0$.

A complex principal $H$--bundle on a Sasakian manifold ${\mathcal X}\, :=\, (X,\, g,\, \xi)$ is
a principal $H$--bundle $(E_H,\, p,\, J_E)$ on $X$, equipped with a
partial connection $\widetilde{\xi}$ for the one dimensional distribution
${\mathbb R}\cdot\xi$ on $X$, such that
$$
L_{\widetilde{\xi}} J_E\,=\, 0\, .
$$

\subsection{Holomorphic principal bundles}

For a connected Sasakian manifold ${\mathcal X}\, :=\, (X,\, g,\, \xi)$ of
dimension $2m+1$, consider
$F$ equipped with the almost complex structure $\Phi\vert_F$ (see \eqref{e0}). Define
\begin{equation}\label{e1}
F^{0,1}\, :=\, \{v+\sqrt{-1}\cdot \Phi(v)\, \mid\, v\, \in\, F\}\, \subset\, F\otimes_{\mathbb R}
\mathbb C
\end{equation}
and
\begin{equation}\label{e2}
F^{1,0}\, :=\, \{v-\sqrt{-1}\cdot \Phi(v)\, \mid\, v\, \in\, F\}\, \subset\, F\otimes_{\mathbb R}
\mathbb C\, ,
\end{equation}
so $F\otimes_{\mathbb R}{\mathbb C}\,=\, F^{1,0}\oplus F^{0,1}$ and
$F^{1,0}\,=\, \overline{F^{0,1}}$. We note that $F^{1,0}$ is a holomorphic Hermitian vector bundle
over $\mathcal X$ (see \cite[p.~551, Definition 3.3]{BS} and \cite[p.~551, Definition 3.5]{BS}
for definition of holomorphic Hermitian bundles over a Sasakian manifold). The partial connection
on $F^{1,0}$ defining its holomorphic structure is given by the Lie bracket operation of vector
fields; the Hermitian structure on $F^{1,0}$ is given by the Riemannian metric $g$ on $X$.

The complex distribution on $X$
\begin{equation}\label{e3}
{\mathcal F}^{0,1} \,:=\, F^{0,1}\oplus {\mathbb C}\cdot \xi \, \subset\,
TX\otimes_{\mathbb R} \mathbb C
\end{equation}
of complex dimension $m+1$ is integrable \cite[p.~550, Lemma 3.4]{BS}.

Let $H$ be a complex Lie group.
Let $p\, :\, E_H\, \longrightarrow\, X$ be a principal $H$--bundle on $X$;
as before, the almost complex structure on the fibers of $E_H$ will be denoted by
$J_E$. Let
$\widetilde{\xi}\, \subset\, TE_H$ be a distribution such that 
$(E_H,\, p,\, J_E,\, \widetilde{\xi})$ is a complex principal $H$--bundle on ${\mathcal X}$. A {\it
holomorphic structure} on $(E_H,\, p,\, J_E,\, \widetilde{\xi})$ is a complex distribution
$$
\widetilde{\mathcal F}\, \subset\, TE_H\otimes_{\mathbb R}{\mathbb C}
$$
of complex dimension $m+1$ satisfying the following five conditions:
\begin{enumerate}
\item the distribution $\widetilde{\mathcal F}$ is integrable,

\item $\widetilde{\xi}\, \subset\, \widetilde{\mathcal F}$,

\item the complexified differential $$dp\otimes_{\mathbb R}{\mathbb C}\, :\, TE_H\otimes_{\mathbb R}{\mathbb C}
\,\longrightarrow\, p^* TX\otimes_{\mathbb R}{\mathbb C}$$ of the projection $p$
restricts to an isomorphism between $\widetilde{\mathcal F}$ and $p^*{\mathcal F}^{0,1}$
(defined in \eqref{e3}),

\item the action of $H$ on $TE_H\otimes_{\mathbb R}{\mathbb C}$, given by the action of $H$ on $E_H$,
preserves the subbundle $\widetilde{\mathcal F}$, and

\item the complex distribution $\widetilde{\mathcal F}\oplus {\mathbb K}^{0,1}$
(see \eqref{f01}) on $E_H$ is integrable.
\end{enumerate}
In particular, $\widetilde{\mathcal F}$ is a partial connection on $E_H$ for the
complex distribution ${\mathcal F}^{0,1}$ on $X$.

A holomorphic principal $H$--bundle on $\mathcal X$ is a complex principal $H$--bundle on
${\mathcal X}$ equipped with a holomorphic structure.

Let
\begin{equation}\label{rho}
\rho\,:\, H\, \longrightarrow\, Q
\end{equation}
be a holomorphic homomorphism of complex
Lie groups. Take a holomorphic
principal $H$--bundle
$(E_H,\, p,\, J_E,\, \widetilde{\xi},\, \widetilde{\mathcal F})$
on $\mathcal X$. Let
$$
E_Q\, :=\, E_H\times^\rho Q\, \stackrel{p'}{\longrightarrow}\, X
$$
be the principal $Q$--bundle over $X$ obtained by extending the structure group of $E_H$
using the above homomorphism $\rho$. The complex structure $J_E$ on the fibers of
$E_H$ and the complex structure of $Q$ together produce a complex structure $J_{E_Q}$ on the
fibers of $E_Q$. As noted before, a partial connection on $E_H$ produces a partial connection
on $E_Q$. Therefore, the partial connections $\widetilde{\xi}$ and
$\widetilde{\mathcal F}$ produce partial connections on $E_Q$ for the distributions $\xi$ and
${\mathcal F}^{0,1}$ respectively. Let $\widetilde{\xi}_Q$ and $\widetilde{\mathcal F}_Q$ 
denote these partial connections on $E_Q$ for the distributions $\xi$ and
${\mathcal F}^{0,1}$ respectively. Now it is straight-forward to check that
\begin{equation}\label{r2}
(E_Q,\, p',\, J_{E_Q},\, \widetilde{\xi}_Q,\, \widetilde{\mathcal F}_Q)
\end{equation}
is a holomorphic principal $Q$--bundle on $\mathcal X$.

Set $H\,=\, {\rm GL}(r, {\mathbb C})$. Take a holomorphic
principal ${\rm GL}(r, {\mathbb C})$--bundle
$(E,\, p,\, J_E,\, \widetilde{\xi},\, \widetilde{\mathcal F})$ 
on $\mathcal X$. Then the vector bundle $E\times^{{\rm GL}(r, {\mathbb C})} {\mathbb C}^r$
associated to $E$ for the standard representation of ${\rm GL}(r, {\mathbb C})$
on ${\mathbb C}^r$ is a
holomorphic vector bundle over $\mathcal X$ (see \cite{BS} for
holomorphic vector bundles on Sasakian manifolds). Conversely, if $V$ is a holomorphic
vector bundle on $\mathcal X$ of rank $r$, then the frame bundle for $V$ is a holomorphic
principal ${\rm GL}(r, {\mathbb C})$--bundle over $\mathcal X$.

For a general complex Lie group $H$, assume that $Q$ in \eqref{rho} is ${\rm GL}(r, {\mathbb 
C})$. Consequently, $(E_Q,\, p,\, J_{E_Q},\, \widetilde{\xi}_Q,\, \widetilde{\mathcal F}_Q)$ in 
\eqref{r2} produces a holomorphic vector bundle on $\mathcal X$ of rank $r$.

Now take $\rho$ to be the adjoint representation of $H$ in its Lie algebra $\text{Lie}(H)$. Then 
the corresponding holomorphic vector bundle $(E_Q,\, p',\, J_{E_Q},\, \widetilde{\xi}_Q,\, 
\widetilde{\mathcal F}_Q)$ on $\mathcal X$ will be called the adjoint bundle for $E_H$, and it 
will be denoted by $\text{ad}(E_H)$. We note that
\begin{equation}\label{ad}
\text{ad}(E_H)\,=\, {\mathbb K}^{1,0}/H\, ,
\end{equation}
where ${\mathbb K}^{1,0}$ is constructed in \eqref{f02}; the isomorphism in \eqref{ad} is
obtained from the action of $H$ on $E_H$ that identifies ${\mathbb K}^{1,0}$ with the trivial
vector bundle on $E_H$ with fiber $\text{Lie}(H)$.

Given holomorphic vector bundles $W_1$ and $W_2$ on $\mathcal X$, 
holomorphic homomorphisms from $W_1$ to $W_2$ are defined in \cite[p.~551, Definition 3.4]{BS}.
We recall that a holomorphic homomorphism is a fiber-wise $\mathbb C$--linear $C^\infty$ map
\begin{equation}\label{Psi}
\Psi\, :\, W_1\, \longrightarrow\, W_2
\end{equation}
such that $\Psi$ intertwines the partial connections on $W_1$ and $W_2$ in the direction
of the distribution ${\mathcal F}^{0,1}$ defined in \eqref{e3}.

\section{Holomorphic connections on a principal bundle}

As before, ${\mathcal X}\, :=\, (X,\, g,\, \xi)$ is a connected Sasakian manifold of 
dimension $2m+1$. Consider the complex vector bundle $F^{1,0}\, \longrightarrow\, X$ in 
\eqref{e2}; the complex structure on the fibers of it is given by
$\Phi\vert_F$ in \eqref{e0}. As noted before, $F^{1,0}$ is 
a holomorphic vector bundle on $\mathcal X$ of complex rank $m$; the partial connections
are given by Lie derivative.

Let $(E_H,\, p,\, J_E,\, \widetilde{\xi},\, \widetilde{\mathcal F})$ be a holomorphic
principal $H$--bundle on $\mathcal X$. Consider the subbundle
$F^{0,1}\, \subset\, {\mathcal F}^{0,1}$ in \eqref{e3}. Recall that the complexified differential
$dp\otimes_{\mathbb R}{\mathbb C}\, :\, TE_H\otimes_{\mathbb R}{\mathbb C}
\,\longrightarrow\, p^*TX\otimes_{\mathbb R}{\mathbb C}$ of the projection $p$
restricts to an isomorphism between $\widetilde{\mathcal F}$ and $p^*{\mathcal F}^{0,1}$.
Let $${\mathbb F}\, \subset\, \widetilde{\mathcal F}$$
be the subbundle corresponding to
$p^*F^{0,1}$ under this isomorphism between $\widetilde{\mathcal F}$ and $p^*{
\mathcal F}^{0,1}$ given by $dp\otimes_{\mathbb R}{\mathbb C}$. Now define the complex distribution
\begin{equation}\label{bt}
{\mathbb T}\, :=\, 
{\mathbb K}^{1,0}\oplus \overline{\mathbb F}\, \subset\, TE_H\otimes_{\mathbb R}{\mathbb C}
\end{equation}
on $E_H$, where ${\mathbb K}^{1,0}$ is defined in \eqref{f02}. Since the complexified differential 
$dp\otimes_{\mathbb R}{\mathbb C}$ gives an isomorphism between ${\mathbb F}$ and $p^*F^{0,1}$, 
it also gives an isomorphism between $\overline{\mathbb F}$ and $p^*F^{1,0}$ (defined in 
\eqref{e2}). Note that the action of $H$
on $TE_H\otimes_{\mathbb R}{\mathbb C}$, given by the action of $H$ on $E_H$, preserves this
subbundle $\mathbb T$ in \eqref{bt}. The quotient
\begin{equation}\label{at}
\text{At}(E_H)\, :=\, {\mathbb T}/H \, \longrightarrow\, E_H/H \,=\, X
\end{equation}
is a complex vector bundle. Although the complex distribution $\mathbb T$ is not integrable,
it is straightforward to check that the complex distribution ${\mathbb T}\oplus {\mathbb C}\cdot
\widetilde{\xi}$ on $E_H$ is integrable. Using this it follows that $\text{At}(E_H)$ is a
holomorphic vector bundle over the Sasakian manifold $\mathcal X$. This $\text{At}(E_H)$ will
be called the \textit{Atiyah bundle} for $E_H$. In view of \eqref{bt}, comparing \eqref{ad} and
\eqref{at} it follows that
\begin{equation}\label{ad2}
\text{ad}(E_H)\, \subset\, \text{At}(E_H)\, ;
\end{equation}
this inclusion of $\text{ad}(E_H)$ in $\text{At}(E_H)$ is holomorphic. We shall now investigate 
the quotient vector bundle $\text{At}(E_H)/\text{ad}(E_H)$. 

The isomorphism $(dp\otimes_{\mathbb R}{\mathbb C})\vert_{\overline{\mathbb F}} $ between $\overline{\mathbb F}$ and $p^*F^{1,0}$ and the zero homomorphism ${\mathbb K}^{1,0}\, \longrightarrow\, 
p^*F^{1,0}$ together produce a homomorphism
$$
0\oplus (dp\otimes_{\mathbb R}{\mathbb C})\vert_{\overline{\mathbb F}}\,:\,
{\mathbb T}\, :=\,
{\mathbb K}^{1,0}\oplus \overline{\mathbb F} \, \longrightarrow\, p^*F^{1,0}\, .
$$
The above homomorphism $0\oplus (dp\otimes_{\mathbb R}{\mathbb C})\vert_{\overline{\mathbb F}}$
is $H$--equivariant, and hence it descends to a homomorphism
$$
\text{At}(E_H)\, :=\, {\mathbb T}/H \, \stackrel{d'p}{\longrightarrow}\, (p^*F^{1,0})/H\,=\,
F^{1,0}\, .
$$
This descended homomorphism $d'p$ is evidently surjective, and for the
holomorphic subbundle $\text{ad}(E_H)$ in \eqref{ad2} we have
$$
d'p(\text{ad}(E_H))\,=\, 0\, .
$$
We have the following short exact sequence of holomorphic vector bundles on
$\mathcal X$:
\begin{equation}\label{at2}
0\, \longrightarrow\, \text{ad}(E_H) \, \longrightarrow\, \text{At}(E_H)
\,\stackrel{d'p}{\longrightarrow} \, F^{1,0} \, \longrightarrow\, 0\, ;
\end{equation}
holomorphic homomorphisms of holomorphic vector bundles on $\mathcal X$ are defined
in \eqref{Psi}.
The short exact sequence in \eqref{at2} will be called the \textit{Atiyah exact sequence} for $E_H$.

A \textit{holomorphic connection} on $E_H$ is a holomorphic homomorphism of vector bundles
on $\mathcal X$
$$
D\, :\, F^{1,0}\, \longrightarrow\,\text{At}(E_H)
$$
such that $(d'p) \circ D\,=\, \text{Id}_{F^{1,0}}$, where $d'p$ is the projection
in \eqref{at2}.

The above definition of a holomorphic connection on $E_H$ is modeled on the definition of
a holomorphic connection on a holomorphic principal bundle over a complex manifold (see \cite{At}).

Using the orthogonal splitting of the real tangent bundle
$$
TX\,=\, {\mathbb R}\cdot\xi \oplus \xi^\perp \,=\, {\mathbb R}\cdot\xi \oplus F\, ,
$$
we consider $F^*$ as a subbundle of the real cotangent bundle of $X$. So
$(F^{1,0})^*$ is a subbundle of $(TX)^*\otimes_{\mathbb R}\mathbb C$.

\begin{lemma}\label{lem1}
Let $D$ be a holomorphic connection on $E_H$. Then $D$ defines a usual connection $\widetilde
D$ on the $C^\infty$ principal
$H$--bundle $E_H$. The curvature of $\widetilde D$ is a $C^\infty$ section of
the vector bundle ${\rm ad}(E_H)\otimes\bigwedge^2 (F^{1,0})^*$.
\end{lemma}

\begin{proof}
The homomorphism $D$ gives a partial connection on $E_H$ in the direction of $F\,=\,\xi^{\perp}$.
So $D$ and the given partial connection on $E_H$ in the direction of $\xi$ together
produce a usual connection on $E_H$. The curvature of this connection is evidently a section of
the vector bundle ${\rm ad}(E_H)\otimes\bigwedge^2 (F^{1,0})^*$.
\end{proof}

The Reeb vector field $\xi$ on $X$ defines a flow on $X$ while the vector field $\widetilde \xi$ 
on $E_H$ defines a flow on $E_H$. The projection $p\,:\, E_H\, \longrightarrow \, X$ intertwines 
these two flows. The curvature of the connection $\widetilde D$ in Lemma \ref{lem1} is actually 
preserved by this flow.

\begin{definition}\label{def0}
A holomorphic connection $D$ the a holomorphic principal $H$--bundle $E_H$ on the
Sasakian manifold $\mathcal X$ will be called {\it flat} if the curvature of the 
corresponding usual connection $\widetilde D$ in Lemma \ref{lem1} vanishes identically.
\end{definition}

Let $\rho\, :\, H\, \longrightarrow\, Q$ be a holomorphic homomorphism of complex Lie groups.
Consider the holomorphic principal $Q$--bundle
$(E_Q,\, p',\, J_{E_Q},\, \widetilde{\xi}_Q,\, \widetilde{\mathcal F}_Q)$ on $\mathcal X$
constructed in \eqref{r2} from the holomorphic principal $H$--bundle $E_H$. Let
$$
\rho'\, :\, \text{ad}(E_H)\, \longrightarrow\, \text{ad}(E_Q)
$$
be the homomorphism of holomorphic vector bundles given by the homomorphism of Lie
algebras corresponding to the above homomorphism $\rho$ of Lie groups. Consider
the injective homomorphism of holomorphic vector bundles
$$
\text{ad}(E_H)\, \longrightarrow\, \text{At}(E_H)\oplus \text{ad}(E_Q)\, ,\ \
v\, \longmapsto\, (-v,\, \rho'(v))\, ;
$$
the above inclusion map $\text{ad}(E_H)\, \hookrightarrow\, \text{At}(E_H)$ is the one in
\eqref{at2}. Then we have
$$
\text{At}(E_Q)\,=\, (\text{At}(E_H)\oplus \text{ad}(E_Q))/\text{ad}(E_H)\, .
$$
If $D\, :\, F^{1,0}\, \longrightarrow\,\text{At}(E_H)$ is a holomorphic connection on $E_H$,
then the homomorphism
$$
F^{1,0}\, \longrightarrow\, \text{At}(E_H)\oplus \text{ad}(E_Q)\, , \ \
v\, \longmapsto\, (D(v),\, 0)
$$
descends to a homomorphism $F^{1,0}\, \longrightarrow\, \text{At}(E_Q)$ that defines
a holomorphic connection on the principal $Q$--bundle $E_Q$ over $\mathcal X$; it is
called the holomorphic connection on $E_Q$ induced by $D$.

From Lemma \ref{lem1} we know that the above holomorphic connection on $E_Q$ induced by $D$
defines a usual connection on $E_Q$. This (usual) connection on $E_Q$ clearly coincides with the
connection on $E_Q$ induced by the connection $\widetilde D$ in Lemma \ref{lem1} given by $D$.

\section{Branched holomorphic Cartan geometry}

\subsection{Definitions}

Let $G$ be a connected complex Lie group and $H\, \subset\, G$ a complex Lie subgroup. The Lie 
algebras of $G$ and $H$ will be denoted by $\mathfrak g$ and $\mathfrak h$ respectively. Take 
a holomorphic principal $H$--bundle $${\mathcal E}_H\, =\, (E_H,\, p,\, J_E,\, 
\widetilde{\xi},\, \widetilde{\mathcal F})$$ on $\mathcal X$. Let $${\mathcal E}_G\, =\, 
(E_G,\, p',\, J'_E,\, \widetilde{\xi}',\, \widetilde{\mathcal F}')$$ be the holomorphic 
principal $G$--bundle on $\mathcal X$ obtained by extending the structure group of ${\mathcal 
E}_H$ using the inclusion of $H$ in $G$. The inclusion of $\mathfrak h$ in $\mathfrak g$, being
$H$--equivariant, produces an inclusion of $\text{ad}({\mathcal E}_H)$ in
$\text{ad}({\mathcal E}_H)$; this
inclusion map is holomorphic. We have the following two short exact sequences of holomorphic
vector bundles on $\mathcal X$ with a common first term:
\begin{equation}\label{j1}
0\, \longrightarrow\, \text{ad}({\mathcal E}_H) \,\stackrel{\iota_1}{\longrightarrow}\,
\text{ad}({\mathcal E}_G)
\,\longrightarrow \, \text{ad}({\mathcal E}_G)/\text{ad}({\mathcal E}_H) \, \longrightarrow\, 0
\end{equation}
and
\begin{equation}\label{j2}
0\, \longrightarrow\, \text{ad}({\mathcal E}_H) \, \stackrel{\iota_2}{\longrightarrow}\,
\text{At}({\mathcal E}_H) \,\stackrel{d'p}{\longrightarrow} \, F^{1,0} \, \longrightarrow\, 0
\end{equation}
(see \eqref{at2}).

\begin{definition}\label{def1}
A {\it branched holomorphic Cartan geometry} on $\mathcal X$ of type $(G,\, H)$ is a holomorphic principal
$H$--bundle ${\mathcal E}_H$ on $\mathcal X$ together with a holomorphic homomorphism of 
$$
\varphi\, :\, \text{At}({\mathcal E}_H) \, \longrightarrow\, \text{ad}({\mathcal E}_G)
$$
satisfying the following two conditions:
\begin{enumerate}
\item $\varphi$ is an isomorphism over a nonempty subset of $X$, and

\item for any $v\, \in\, \text{ad}({\mathcal E}_H)$, the equality
$$
\varphi(\iota_2(v))\,=\, \iota_1(v)
$$
holds, where $\iota_1$ and $\iota_2$ are the homomorphisms in \eqref{j1} and \eqref{j2} respectively.
\end{enumerate}
If $\varphi$ is an isomorphism over $X$, then the pair $({\mathcal E}_H,\, \varphi)$ is
called a {\it holomorphic Cartan geometry}.
\end{definition}

The {\it branching locus} of a branched holomorphic Cartan geometry $({\mathcal E}_H,\, \varphi)$
is the subset of $X$ where $\varphi$ fails to be an isomorphism.

From Definition \ref{def1} it follows immediately that a branched holomorphic Cartan geometry
$({\mathcal E}_H,\, \varphi)$ produces the following commutative diagram of holomorphic
homomorphisms of vector bundles on $\mathcal X$
\begin{equation}\label{j3}
\begin{matrix}
0 &\longrightarrow & \text{ad}({\mathcal E}_H) &\longrightarrow & \text{At}({\mathcal E}_H)
& \stackrel{d'p}{\longrightarrow} & F^{1,0} &\longrightarrow & 0\\
&& \Vert && \,~\,\Big\downarrow\varphi && \,~\,\Big\downarrow\phi\\
0 &\longrightarrow & \text{ad}({\mathcal E}_H) &\longrightarrow & \text{ad}({\mathcal E}_G)
&\longrightarrow & \text{ad}({\mathcal E}_G)/\text{ad}({\mathcal E}_H)
&\longrightarrow & 0
\end{matrix}
\end{equation}
where $\phi$ is induced by $\varphi$. We note that $\phi$ is an isomorphism over a point
$x\, \in\, X$ if and only if $\varphi$ is an isomorphism over $x$. Therefore,
$\phi$ is an isomorphism over a nonempty subset of $X$. This nonempty subset is evidently
open and dense. If $({\mathcal E}_H,\, \varphi)$ is a holomorphic Cartan geometry, then $\phi$ is
an isomorphism over $X$.

\subsection{Standard examples}\label{se4.2}

Take a connected complex Lie group $G$ and a complex Lie subgroup
$H\, \subset\, G$. Assume that $G/H$ is equipped with a K\"ahler form $\omega$ satisfying
the following condition: there is a holomorphic line bundle $L$ over $G/H$ equipped with a
Hermitian structure $h_L$ such that the curvature of the Chern connection on $L$ for $h_L$ coincides
with $\sqrt{-1}\cdot\omega$. Consider the real hypersurface
$$
X_0\, :=\, \{v\, \in\, L\, \mid\, h_L(v)\,=\, 1\}\, \subset\, L
$$
in the total space of $L$. Let
\begin{equation}\label{q}
q\, :\, X_0\, \longrightarrow\, G/H
\end{equation}
be the natural projection; note that $q$ makes $X_0$ a principal $S^1$--bundle over $G/H$.
Then $\omega$ and the Chern connection for $h_L$ together produce a 
Sasakian structure on $X_0$. The Chern connection for $h_L$ decomposes the real tangent
bundle $T X_0$ of $X_0$ as
$$
T X_0\,=\, q^* T(G/H) \oplus ((G/H)\times {\mathbb R})\, ,
$$
where $T(G/H)$ is the real tangent bundle of $G/H$; more precisely, $q^* T(G/H)$ is the 
horizontal tangent space and $(G/H)\times {\mathbb R}$ is the vertical tangent space for the 
connection on the $S^1$--bundle $X_0$ given by the Chern connection on $L$ for $h_L$. The vector 
field on $X_0$ given by the action of $S^1$ on $X_0$ is the Reeb 
vector field $\xi$. So $X_0$ is a regular Sasakian manifold. We will denote by ${\mathcal X}_0$ 
this manifold $X_0$ equipped with the Sasakian structure.

The quotient map $G\, \longrightarrow\, G/H$ defines a holomorphic principal $H$--bundle over
$G/H$. This holomorphic principal $H$--bundle over $G/H$ will be denoted by $E^1_H$. Let
$$E^1_G\, := \,E^1_H\times^H G \, \longrightarrow\, G/H$$ be the holomorphic principal $G$--bundle
over $G/H$ obtained by extending the structure group of $E^1_H$ using the inclusion of $H$ in $G$. We note
that $E^1_G$ is identified with the trivial principal $G$--bundle $(G/H)\times G\,\longrightarrow\,
G/H$. Indeed, $E^1_G$ is the quotient of $G\times G$ where two elements $(g_1,\, g'_1)$ and
$(g_2,\, g'_2)$ are identified if there is an element $h\, \in\, H$ such that $g_2\,=\, g_1h$ and
$g'_2 \,=\, h^{-1}g'_1$. So the self-map of $G\times G$ defined by
$(g_1,\, g'_1)\, \longmapsto\, (g_1,\, g_1g'_1)$ identifies $E^1_G$ with 
the trivial principal $G$--bundle $(G/H)\times G$ over $G/H$.
Identify the holomorphic tangent bundle $TG$ with the trivial holomorphic
bundle $G\times {\mathfrak g}$, where $\mathfrak g$ is the Lie algebra of $G$, using
right-invariant vector fields on $G$. This trivialization produces an isomorphism
$$\mu\, :\, \text{At}(E^1_H)\, \longrightarrow\, \text{ad}(E^1_G)$$ using the
isomorphisms
$$
\text{At}(E^1_H)\,:=\, (TG)/H \, =\, (G\times {\mathfrak g})/H \,=\, (G/H)\times {\mathfrak g}
\,=\, \text{ad}(E^1_G)\, ;
$$
recall that $E^1_G$ is the trivial principal $G$--bundle $(G/H)\times G\, \longrightarrow\,G/H$, 
so $\text{ad}(E^1_G)$ is the trivial vector bundle $(G/H)\times {\mathfrak g}$ over $G/H$. Hence 
$E^1_H$ and the above isomorphism $\mu$ together produce a tautological holomorphic Cartan 
geometry on $G/H$ of type $(G,\, H)$ \cite{Sh}.

Consider the pull back $(q^*E^1_H,\, q^*\mu)$, where $q$ is the projection in \eqref{q}. This 
pair defines a holomorphic Cartan geometry of type $(G,\, H)$ on the Sasakian manifold 
${\mathcal X}_0$ constructed above.

Let ${\mathcal X}\, :=\, (X,\, g,\, \xi)$ be a Sasakian manifold, and let
\begin{equation}\label{f}
f\, :\, {\mathcal X}\, \longrightarrow\, G/H
\end{equation}
be a holomorphic map. We recall that the
holomorphicity of $f$ means that the differential
$$
df\, :\, TX\, \longrightarrow\, f^*T(G/H)
$$
(here $T(G/H)$ is the real tangent bundle of $G/H$) satisfies the following two conditions:
\begin{itemize}
\item $df (\xi)\, =\, 0$, and 

\item $df$ intertwines the automorphism $\Phi\vert_F$ in \eqref{e0} and the automorphism of
$f^*T(G/H)$ given by the almost complex structure on the complex manifold $G/H$.
\end{itemize}
Assume that the restriction $(df)\vert_F$ is an isomorphism over some point of $X$; this 
implies that $(df)\vert_F$ is an isomorphism over an open dense subset of $X$. Then 
$(f^*E^1_H,\, f^*\mu)$ is a holomorphic branched Cartan geometry on $\mathcal X$ of type $(G,\, 
H)$. The branching locus for this holomorphic branched Cartan geometry $(f^*E^1_H,\, f^*\mu)$ 
is the closed subset of $X$ where $(df)\vert_F$ fails to be an isomorphism.

\subsection{A connection defined by branched Cartan geometry}

Let ${\mathcal X}\, :=\, (X,\, g,\, \xi)$ be a connected Sasakian manifold. Take a pair
$({\mathcal E}_H,\, \varphi)$, where
${\mathcal E}_H\, :=\, (E_H,\, p,\, J_E,\, \widetilde{\xi},\, \widetilde{\mathcal F})$ is a
holomorphic principal $H$--bundle on $\mathcal X$, that defines a branched holomorphic
Cartan geometry on $\mathcal X$ of type $(G,\, H)$. As before, let ${\mathcal E}_G\, =\,
(E_G,\, p',\, J'_E,\, \widetilde{\xi}',\, \widetilde{\mathcal F}')$ be the holomorphic
principal $G$--bundle on $\mathcal X$ obtained by extending the structure group of ${\mathcal
E}_H$ using the inclusion of $H$ in $G$.

\begin{proposition}\label{prop1}
The above holomorphic principal $G$--bundle ${\mathcal E}_G$ has a natural holomorphic
connection given by $\varphi$.
\end{proposition}

\begin{proof}
We will first describe $\text{At}({\mathcal E}_G)$. Consider the homomorphism
$$
\iota_3\, :\, \text{ad}({\mathcal E}_H) \,\longrightarrow\,
\text{ad}({\mathcal E}_G)\oplus \text{At}({\mathcal E}_H)\, ,\ \ v\, \longmapsto\,
(\iota_1(v),\, -\iota_2(v))\, ,
$$
where $\iota_1$ and $\iota_2$ are the homomorphisms in \eqref{j1} and \eqref{j2} respectively.
Then we have
$$
(\text{ad}({\mathcal E}_G)\oplus \text{At}({\mathcal E}_H))/\iota_3(\text{ad}({\mathcal E}_H))
\,=\, \text{At}({\mathcal E}_G)\, .
$$
Indeed, this follows immediately from the construction of the Atiyah bundle (see
\eqref{at}). Let
$$
\alpha\, :\, \text{ad}({\mathcal E}_G)\oplus \text{At}({\mathcal E}_H)\,\longrightarrow\,
(\text{ad}({\mathcal E}_G)\oplus \text{At}({\mathcal E}_H))/\iota_3(\text{ad}({\mathcal E}_H))
\,=\, \text{At}({\mathcal E}_G)
$$
be the quotient map. Let
\begin{equation}\label{ip}
0\, \longrightarrow\, \text{ad}({\mathcal E}_G) \,\stackrel{\iota'}{\longrightarrow}\,
\text{At}({\mathcal E}_G) \,\stackrel{d'p'}{\longrightarrow} \, F^{1,0} \, \longrightarrow\, 0
\end{equation}
be the Atiyah exact sequence for ${\mathcal E}_G$ (see \eqref{at2}). The injective
homomorphism $\iota'$ in \eqref{ip} coincides $\alpha\circ j_1$, where $j_1$ is the inclusion of
$\text{ad}({\mathcal E}_G)$ in $\text{ad}({\mathcal E}_G)\oplus \text{At}({\mathcal E}_H)$ and
$\alpha$ is the above quotient map. Now consider the homomorphism
$$
\beta\, :\, \text{ad}({\mathcal E}_G)\oplus \text{At}({\mathcal E}_H)\,\longrightarrow\,
\text{ad}({\mathcal E}_G)\, , \ \ (v,\, w) \, \longmapsto\, v+\varphi(w)\, ,
$$
where $\varphi$ is the homomorphism in the statement of the proposition.
Since $\beta\circ\iota_3\, =\, 0$, it follows that $\beta$ descends to a homomorphism
$$
\beta'\, :\, \text{At}({\mathcal E}_G) \, \longrightarrow\, \text{ad}({\mathcal E}_G)\, .
$$
Now it is straightforward to check that $\beta'\circ\iota' \,=\, \text{Id}_{\text{ad}({\mathcal
E}_G)}$, where $\iota'$ is the homomorphism in \eqref{ip}. Consequently, $\beta'$ produces a
holomorphic splitting of the short exact sequence in \eqref{ip}. Hence $\beta'$ gives
a holomorphic connection on ${\mathcal E}_G$.
\end{proof}

\begin{definition}\label{def2}
A branched holomorphic Cartan geometry $({\mathcal E}_H,\, \varphi)$ of type $(G,\, H)$
on $\mathcal X$ is called {\it flat} if the holomorphic connection on
${\mathcal E}_G$ in Proposition \ref{prop1} is flat (see Definition \ref{def0}).
\end{definition}

The branched holomorphic Cartan geometries in Section \ref{se4.2} are flat.

\subsection{Developing map for flat Cartan geometries}\label{se-dm}

Let $({\mathcal E}_H,\, \varphi)$ be a flat branched holomorphic Cartan geometry of type $(G,\, H)$
on $\mathcal X$. Consider the holomorphic connection on ${\mathcal E}_G$ in Proposition \ref{prop1}.
Let $\widehat D$ be the flat connection on $E_G$ given by it in Lemma \ref{lem1}.

Now assume that the manifold $X$ is simply connected. Fix a point $x_0\, \in\, X$. Using the 
flat connection $\widehat D$ on $E_G$, the principal $G$--bundle $E_G$ gets identified with the 
trivial principal $G$--bundle $X\times (E_G)_{x_0}$, where $(E_G)_{x_0}$ is the fiber of $E_G$ 
over the base point $x_0$. This identification between $E_G$ and the trivial principal 
$G$--bundle $X\times (E_G)_{x_0}$ is clearly holomorphic. So we have
$$
E_H\, \subset\, E_G\,=\, E_G\times (E_G)_{x_0}\, .
$$
Let
\begin{equation}\label{ga}
\gamma\, :\, X\, \longrightarrow\, (E_G)_{x_0}/H
\end{equation}
be the map that sends any $x\, \in\, X$ to the $H$--orbit $(E_H)_x\, \subset\, (E_G)_{x_0}$, 
where $(E_H)_x$ is the fiber of $E_H$ over the point $x$. Since the identification between $E_G$ 
and the trivial principal $G$--bundle $X\times (E_G)_{x_0}$ is holomorphic with respect to the 
holomorphic structure on $E_G$ given by the holomorphic structure on $E_H$, it follows 
immediately that the map $\gamma$ in \eqref{ga} is holomorphic. In particular, $\gamma$ is
constant on the orbits of the flow on $X$ given by the Reeb vector field $\xi$.

Fixing a point $y_0\, \in\, (E_G)_{x_0}$, we may identify $G$ with $(E_G)_{x_0}$ by the
map $g\, \longmapsto\, y_0g$. In that case, $\gamma$ is a holomorphic map from
$X$ to $G/H$. If we set $f$ in \eqref{f} to be the map $\gamma$ in \eqref{ga}, then
the pulled back holomorphic branched Cartan geometry $(f^*E^1_H,\, f^*\mu)$ in
Section \ref{se4.2} is identified with the holomorphic branched Cartan geometry
$({\mathcal E}_H,\, \varphi)$ that we started with.

The map $\gamma$ in \eqref{ga} will be called the \textit{developing map}
for $({\mathcal E}_H,\, \varphi)$.

\begin{lemma}\label{lem2}
Let ${\mathcal X}\, :=\, (X,\, g,\, \xi)$ be a connected Sasakian manifold such that
$X$ is compact and simply connected. Then there is no flat branched holomorphic
Cartan geometry on $\mathcal X$ of type $(G,\, H)$ if $G/H$ is noncompact.
\end{lemma}

\begin{proof}
Since $X$ is compact and connected while $G/H$ is noncompact, there is no
holomorphic map from $\mathcal X$ to $G/H$ satisfying the condition that
its restriction to some nonempty open subset of $X$ is a submersion. Hence there is
no developing map $\gamma$ as in \eqref{ga}.
\end{proof}

\section{Quasi-regular Sasakians that are Calabi-Yau}

Let ${\mathcal X}\, :=\, (X,\, g,\, \xi)$ be a connected compact quasi-regular Sasakian manifold
of dimension $2m+1$ satisfying the following condition: the holomorphic line bundle $\bigwedge^m 
F^{1,0}$ on $\mathcal X$ (see \eqref{e2}) admits a holomorphic connection.

Since $\mathcal X$ is compact and quasi-regular, the space of orbits for the Reeb flow $\xi$ has 
the structure of a smooth compact K\"ahler orbifold of complex dimension $m$ \cite{BG2}. We will 
denote by $\mathbb X$ this compact K\"ahler orbifold of complex dimension $m$. The above
condition that the holomorphic line bundle $\bigwedge^m
F^{1,0}$ on $\mathcal X$ in \eqref{e2} admits a holomorphic connection is equivalent to the condition
that $c_1(\mathbb X)\,=\, 0$. This condition implies that $\mathbb X$ admits a Ricci flat
K\"ahler metric; this was conjectured by Calabi and it was proved by Yau in \cite{Ya}, and
for orbifolds it was proved in \cite{Ca}.

Giving a branched holomorphic Cartan geometry on ${\mathcal X}$ of type $(G,\, H)$ is equivalent
to giving a branched holomorphic Cartan geometry of type $(G,\, H)$ on the orbifold
$\mathbb X$. See \cite{BD} for branched holomorphic Cartan geometries on a complex manifold;
the definition in \cite{BD} extends to smooth orbifolds in a straightforward way.

\begin{theorem}\label{thm1}
Assume that the orbifold fundamental group of $\mathbb X$ is trivial.
Let $E$ be a holomorphic vector bundle over the orbifold $\mathbb X$ admitting a holomorphic
connection $D$. Then
\begin{enumerate}
\item the holomorphic vector bundle $E$ is holomorphically trivial, and

\item $D$ is the trivial connection on the trivial holomorphic vector bundle $E$.
\end{enumerate}
\end{theorem}

\begin{proof}
This was proved in \cite{BD} for compact simply connected K\"ahler manifolds $M$ with $c_1(M) 
\,=\, 0$ (see \cite[Theorem 6.2]{BD}). The proof given in \cite{BD} extends to the case of 
smooth compact orbifolds with vanishing $c_1$ once some straightforward modifications are 
incorporated. The main point to note is that the results from \cite{Si1} used in \cite[Theorem 
6.2]{BD} remain valid compact quasi-regular Sasakian manifolds \cite{BM}.
(See also \cite{Si2}.)
\end{proof}

\begin{proposition}\label{prop2}
Assume that the orbifold fundamental group of $\mathbb X$ is trivial. Let 
$({\mathcal E}_H,\, \varphi)$ be a branched holomorphic Cartan geometry of type $(G,\, H)$, with $G$ a complex affine Lie group,
on $\mathcal X$. Then there is a holomorphic map $\gamma\, :\, {\mathcal X}\,\longrightarrow\,
G/H$ such that $({\mathcal E}_H,\, \varphi)$ is the pullback, by $\gamma$, of the standard
Cartan geometry on $G/H$ of type $(G,\, H)$. Also, $G/H$ is compact.
\end{proposition}

Recall that $G$ is a complex affine Lie group if there exist a positive integer $r$ and 
holomorphic homomorphism of complex Lie groups $\rho \,:\, G \,\longrightarrow\, {\rm GL}(r, 
{\mathbb C})$, with discrete kernel. The corresponding Lie algebra representation $\rho'$ is an 
injective Lie algebra homomorphism from $\mathfrak g$ to $\mathfrak{ gl(r, {\mathbb C}})$. 
Notice that complex simply connected Lie groups and complex semi-simple Lie groups are complex 
affine. Indeed, for $G$ complex simply connected, holomorphic representations with discrete 
kernel do exist by Ado's theorem. For $G$ complex semi-simple, holomorphic representations with 
discrete kernel are also known to exist (see Theorem 3.2, chapter XVII in \cite{Hoc}).

\begin{proof}
By Proposition \ref{prop1} the holomorphic principal $G$--bundle ${\mathcal E}_G$ constructed by 
extension of the structure group has a natural holomorphic connection given by $\varphi$.

Let us consider a holomorphic homomorphism of complex Lie groups $\rho \,:\, G 
\,\longrightarrow\, {\rm GL}(r, {\mathbb C})$, with discrete kernel. Then the associated 
holomorphic vector bundle of rank $r$ inherits a holomorphic connection, which must be flat by 
Theorem \ref{thm1}. Since the Lie algebra homomorphism $\rho' \,:\, \mathfrak g 
\,\longrightarrow\, \mathfrak{ gl(r, {\mathbb C}})$ is injective, the curvature of the 
holomorphic connection of the holomorphic principal bundle ${\mathcal E}_G$ also vanishes.

It follows that the branched holomorphic Cartan geometry
$({\mathcal E}_H,\, \varphi)$ is flat (see Definition \ref{def2}). 

Consider the developing map
$$
\gamma\, :\, {\mathcal X}\,\longrightarrow\, G/H
$$
constructed in \eqref{ga}. As observed in Section \ref{se-dm}, the branched holomorphic Cartan 
geometry $({\mathcal E}_H,\, \varphi)$ on $\mathcal X$ is the pullback of the standard Cartan geometry on $G/H$ 
of type $(G,\, H)$ by the map $\gamma$.

{}From Lemma \ref{lem2} we know that $G/H$ is compact.
\end{proof}

%%%%%%%%%%%%%%%%%%%%%%%%%%%%%%%%%%%%%%%%%%%%%%%%%%%%%%%%%%%%%%%%%%%%%%%%%%%%%%%%%%%%%%%%%%

\end{document}